\documentclass{amsart}

\usepackage[utf8]{inputenc}
\usepackage{color}
\usepackage[table]{xcolor}
\usepackage{hyperref}
\usepackage{amssymb}
\usepackage{float}
\usepackage{tikz}
\usepackage{forest}
\usepackage{nicefrac}
\usepackage{enumerate}
\usepackage{tabularx}
\usepackage[nameinlink, capitalize, noabbrev]{cleveref}
\usepackage{listings}
\usetikzlibrary{shapes.geometric,decorations.markings}
\usetikzlibrary{decorations.pathreplacing}
\usetikzlibrary{fit}
\usetikzlibrary{positioning}
\tikzset{mytext/.style={font=\small, text=black}}

\hypersetup{
    colorlinks=true,
    linkcolor=teal,
    citecolor=magenta,
    }

\newtheorem{Theorem}{Theorem}

\newtheorem{proposition}{Proposition}[section]
\newtheorem{lemma}[proposition]{Lemma}

\newtheorem{theorem}[proposition]{Theorem}

\theoremstyle{definition}

\definecolor{blue}{HTML}{008ED7}
\definecolor{mygray}{gray}{0.75}
\definecolor{lightBlue}{rgb}{0.88,1.0,1.0}

\numberwithin{equation}{section}

\title[Hausdorff dimension of regular branch groups]{An algorithm to compute the Hausdorff dimension of regular branch groups}
\author{Jorge Fariña-Asategui}
\address{Jorge Fariña-Asategui: Centre for Mathematical Sciences, Lund University, 223 62 Lund, Sweden -- Department of Mathematics, University of the Basque Country UPV/EHU, 48080 Bilbao, Spain}
\email{jorge.farina\_asategui@math.lu.se}
\thanks{The author is supported by the Spanish Government, grant PID2020-117281GB-I00, partly with FEDER funds. The author also acknowledges support from the Walter Gyllenberg Foundation from the Royal Physiographic Society of Lund}

\begin{document}

\begin{abstract}
An explicit algorithm is given for the computation of the Hausdorff dimension of the closure of a regular branch group in terms of an arbitrary branch structure. We implement this algorithm in GAP and apply it to a family of GGS-groups acting on the 4-adic tree.
\end{abstract}

\maketitle

\section{Introduction}

Groups acting on rooted trees, and in particular branch groups, have been \mbox{thoroughly} studied in the last decades and there are connections even outside of group theory, such as to analysis, dynamics, algebraic geometry, computer science and cryptography; see \cite{Handbook, SelfSimilar} for an overview and \cite{JorgeCyclicity, SukranJone, Crypto, PenlandSunicSubshifts} for some of the aforementioned connections. A well-studied aspect of these groups acting on rooted trees is their Hausdorff dimension; see \cite{JorgeSpectra} and the references therein.

Bartholdi and Noce proved recently in \cite[Corollary E]{MarialauraBartholdi} that the Hausdorff dimension of a regular branch profinite group is computable, i.e. there exists an algorithm to compute its Hausdorff dimension from a given branch structure; see \cref{section: preliminaries} for unexplained terminology here and elsewhere in the introduction. For a branch structure over a level-stabilizer on the $p$-adic tree an explicit formula was given by \v{S}uni\'{c} in \cite[Theorem 4(b)]{SunikHausdorff}. Here we give an explicit formula for an arbitrary branch structure. We note that this algorithm is valid more generally for all generalized regular branch profinite groups as defined in \cite{PenlandSunicSubshifts}. 

Let $\Omega:\mathbb{N}\to \mathbb{N}\cup\{0\}$ be the function counting prime factors with multiplicity. We shall write $\Omega(G:K)$ for $\Omega(|G:K|)$. We also write $G_n:=G/\mathrm{St}_G(n)$ for $n\ge 1$. Finally, recall the sequence $\{s_n(G)\}_{n\ge 1}$ defined by the author in \cite{JorgeSpectra} as
$$s_n(G):=m \log |\mathrm{St}_G(n-1):\mathrm{St}_G(n)|-\log|\mathrm{St}_G(n):\mathrm{St}_G(n+1)|.$$

\begin{Theorem}
\label{Theorem: algorithm Hausdorff dimension}
Let $G\le W_H\le \mathrm{Aut}~T_m$ be a regular branch group branching over a finite-index subgroup $K\le G$. Then the closure~$\overline{G}$ of $G$ in $W_H$ is regular branch over its $\Omega(G:K)$th level stabilizer. Therefore the Hausdorff dimension of $\overline{G}$ in $W_H$ is
$$\mathrm{hdim}_{W_H}(\overline{G})=\frac{1}{\log|H|}\left(\log|G_1|-\sum_{n=1}^{\Omega(G:K)}\frac{s_n(G)}{m^n}\right).$$
Furthermore, the logarithmic indices of the quotients $G_n$ are given by
$$\log|G_n|=\alpha m^n+\beta$$
for all $n\ge \Omega(G:K)$, where 
$$\alpha:=\frac{\mathrm{hdim}_{W_H}(\overline{G})\cdot \log|H|}{m-1}\quad\text{and}\quad \beta:=\frac{\sum_{n=1}^{\Omega(G:K)}s_n-\log|G_1|}{m-1}.$$
\end{Theorem}

This opens the door to computing the Hausdorff dimension of the closure of a regular branch group using only computational methods such as GAP. We implement the algorithm from \cref{Theorem: algorithm Hausdorff dimension} in GAP and use it to compute the Hausdorff dimension of the closure of all not invertible-symmetric GGS-groups acting on the 4-adic tree \cite{GGSElena}; see \cref{Theorem: GGS Hausdorff example}. We also provide a method to discard branch structures on a self-similar group; see \cref{proposition: Second Grigorchuk example} for an example.

\subsection*{\textit{\textmd{Organization}}} 
We present some background material in \cref{section: preliminaries}. In \cref{section: further applications} we prove \cref{Theorem: algorithm Hausdorff dimension} and we conclude by showing an explicit application of \cref{Theorem: algorithm Hausdorff dimension} to GGS-groups in \cref{section: an example}.

\subsection*{\textit{\textmd{Notation}}} All the logarithms will be taken in base $m$, where $m$ is the degree of the regular rooted tree $T_m$.

\section{Preliminaries}
\label{section: preliminaries}

\subsection{Groups acting on rooted trees}

For a natural number $m\ge 2$, we define the \textit{$m$-adic tree} $T_m$ as the rooted tree where each vertex has exactly $m$ immediate descendants. The vertices in~$T_m$ at distance exactly $n$ from the root form the \textit{$n$th level of the tree}. Note that the $m$-adic tree may be identified with the free monoid on $m$ generators. The group of graph automorphisms of the $m$-adic tree will be denoted $\mathrm{Aut}~T_m$. 

For each $n\ge 1$, the pointwise stabilizer of the $n$th level of $T_m$ is called the \textit{$n$th level stabilizer} and it is denoted $\mathrm{St}(n)$. Given an element $f\in \mathrm{Aut}~T_m$ and a vertex $v\in T_m$, we define the \textit{section of $f$ at $v$} as the unique automorphism $f|_v\in \mathrm{Aut}~T_m$ such that
$$(vu)^f=v^fu^{f|_v}$$
for every $u\in T_m$. Writing $1,\dotsc,m$ for the vertices at the first level of $T_m$ we have the isomorphism $\psi:\mathrm{Aut}~T\to (\mathrm{Aut}~T_m\times\overset{m}{\dotsb}\times \mathrm{Aut}~T_m)\rtimes \mathrm{Sym}(m)$ given by
$$f\mapsto (f|_1,\dotsc, f|_m)\sigma_f,$$
where $\sigma_f\in \mathrm{Sym}(m)$ is the permutation action of $f$ on the first level of $T_m$.

Let $G\le \mathrm{Aut}~T_m$ be a subgroup. For each $n\ge 1$, the pointwise stabilizer of the $n$th level of $T_m$ is called the \textit{$n$th level stabilizer of $G$} and denoted $\mathrm{St}_G(n)$. This is a normal subgroup of~$G$ and the corresponding finite quotient $G/\mathrm{St}_G(n)$ is denoted~$G_n$. Note that
\begin{align}
    \label{align: level stabilizers}
    \psi(\mathrm{St}_G(n))=\psi(\mathrm{St}_G(1))\cap \big(\mathrm{St}_G(n-1)\times\overset{m}{\dotsb} \times \mathrm{St}_G(n-1)\big)
\end{align}
for every $n\ge 1$.

We say that the group $G$ is \textit{self-similar} if for every $g\in G$ and every $v\in T_m$ we have $g|_v\in G$, and \textit{weakly self-similar} if this holds for every $g\in \mathrm{St}_G(1)$ and every $v\in T_m$ instead. A subgroup $G\le \mathrm{Aut}~T_m$ is said to be \textit{level-transitive} if it acts transitively on every level of $T_m$. A subgroup $K\le \mathrm{Aut}~T_m$ is called \textit{branching} if 
$$\psi(K)\ge K\times\overset{m}{\dotsb}\times K.$$
A self-similar, level-transitive group $G\le \mathrm{Aut}~T_m$ is called \textit{regular branch} if it contains a finite-index branching subgroup $K\le G$.

\subsection{Hausdorff dimension}

The full automorphism group $\mathrm{Aut}~T_m$ is a profinite group with respect to the \textit{congruence topology}, i.e. the topology where $\{\mathrm{St}(n)\}_{n\ge 1}$ yields a base of open neighbourhoods for the identity. Furthermore $\mathrm{Aut}~T_m$ is a metric space with the distance induced by the filtration $\{\mathrm{St}(n)\}_{n\ge 1}$. As so, one may define a Hausdorff dimension for its closed subsets. In general, if $G\le\mathrm{Aut}~T_m$ is a closed subgroup, the Hausdorff dimension of $G$ in $\mathrm{Aut}~T_m$ is given by
$$\mathrm{hdim}_{\mathrm{Aut}~T_m}(G)=\liminf_{n\to\infty}\frac{\log|G:\mathrm{St}_G(n)|}{\log|\mathrm{Aut}~T_m:\mathrm{St}(n)|}.$$
We say that $G\le \mathrm{Aut}~T_m$ has \textit{strong Hausdorff dimension} if the lower limit above is a proper limit.

For a subgroup $H\le \mathrm{Sym}(m)$, the \textit{iterated wreath product} $W_H\le \mathrm{Aut}~T_m$ is the closed subgroup consisting of automorphisms of $T_m$ whose local action on the immediate descendants of each vertex in $T_m$ is given by a permutation in $H$. Note that $W_{\mathrm{Sym}(m)}=\mathrm{Aut}~T_m$. In general, if $G\le W_H\le \mathrm{Aut}~T_m$ is a closed subgroup we shall compute the Hausdorff dimension of $G$ in $W_H$, which is given by 
$$\mathrm{hdim}_{W_H}(G)=\liminf_{n\to\infty}\frac{\log|G:\mathrm{St}_G(n)|}{\log|W_H:\mathrm{St}_{W_H}(n)|}.$$

For any group $G\le \mathrm{Aut}~T_m$ we define the sequence $\{s_n(G)\}_{n\ge 1}$ via
$$s_n(G):=m \log |\mathrm{St}_G(n-1):\mathrm{St}_G(n)|-\log|\mathrm{St}_G(n):\mathrm{St}_G(n+1)|.$$
An easy computation shows that $s_n(G)=0$ for all $n>k$ if $G$ is regular branch over its $k$th level stabilizer.

This sequence $\{s_n(G)\}_{n\ge 1}$ yields an easy formula for computing the Hausdorff dimension of self-similar groups \cite{JorgeSpectra}. Note that we write $S_G(x)$ for the ordinary generating function $\sum_{n\ge 1}s_n(G)x$.

\begin{theorem}[{see {\cite[Theorem B and Proposition 1.1]{JorgeSpectra}}}]
    \label{theorem: sn formula}
    Let $G\le W_H$ be self-similar, where $H\le \mathrm{Sym}(m)$. Then, the closure of $G$ has strong Hausdorff dimension in $W_H$. What is more, we have
$$\mathrm{hdim}_{W_H}(\overline{G})=\frac{1}{\log |H|}\big(\log|G_1|-S_G(1/m)\big).$$
\end{theorem}

\section{Proof of the main result}
\label{section: further applications}

In this section, we give a constructive proof of \cref{Theorem: algorithm Hausdorff dimension}, which yields an explicit algorithm to compute the Hausdorff dimension of the closure of any regular branch group. The proof is based on the following key observation:

\begin{lemma}
\label{proposition: if trivial always trivial}
Let $G\le \mathrm{Aut}~T_m$ be a weakly self-similar group and a normal branching subgroup $K\le G$. If the quotient $K\mathrm{St}_G(k-1)/K\mathrm{St}_G(k)$ is trivial for some $k\ge 1$ then the quotients $K\mathrm{St}_G(n-1)/K\mathrm{St}_G(n)$ are trivial for all $n\ge k$.
\end{lemma}
\begin{proof}
Note that if $M,N$ are normal subgroups of a group $G$ and $N\le M$, then for any subgroup $H\le G$ we have both
\begin{align}
    \label{align: index inequalities}
    |HM:HN|\le |M:N|\quad\text{and}\quad |H\cap M:H\cap N|\le |M:N|.
\end{align}

Let $L:=\psi^{-1}(K\times\overset{m}{\dotsb}\times K)$ and let us assume $K\mathrm{St}_G(k-1)/K\mathrm{St}_G(k)=1$ for some $k\ge 1$. We prove the result by induction on $n\ge k$. The base case holds by assumption. By the weak self-similarity of $G$ we have an embedding of  $\mathrm{St}_G(n-1)/\mathrm{St}_G(n)$ in 
$$\big(\mathrm{St}_G(n-2)/\mathrm{St}_G(n-1)\big)\times\overset{m}{\dotsb}\times \big(\mathrm{St}_G(n-2)/\mathrm{St}_G(n-1)\big)$$
induced by the isomorphism $\psi$. Note that $L$ is normal in $G$ as $K$ is normal in $G$. Then by induction we obtain 
\begin{align*}
    |K\mathrm{St}_G(n-1): K\mathrm{St}_G(n)|&=|KL\mathrm{St}_G(n-1): KL\mathrm{St}_G(n)| \\
    &\le |L\mathrm{St}_G(n-1): L\mathrm{St}_G(n)|\\
    &=\left|\frac{\psi(L)\big(\psi(\mathrm{St}_G(1))\cap \mathrm{St}_G(n-2)\times\overset{m}{\dotsb}\times \mathrm{St}_G(n-2) \big)}{ \psi(L)\big(\psi(\mathrm{St}_G(1))\cap \mathrm{St}_G(n-1)\times\overset{m}{\dotsb}\times \mathrm{St}_G(n-1) \big)}\right|\\
    &=\left|\frac{\psi(\mathrm{St}_G(1))\cap\psi(L)\big( \mathrm{St}_G(n-2)\times\overset{m}{\dotsb}\times \mathrm{St}_G(n-2) \big)}{ \psi(\mathrm{St}_G(1))\cap\psi(L)\big( \mathrm{St}_G(n-1)\times\overset{m}{\dotsb}\times \mathrm{St}_G(n-1) \big)}\right|\\
    &\le |K\mathrm{St}_G(n-2): K\mathrm{St}_G(n-1)|^m=1,
\end{align*}
where the inequalities follow from \textcolor{teal}{(}\ref{align: index inequalities}\textcolor{teal}{)}, the first equality follows from $K$ being branching, the second equality follows from \cref{align: level stabilizers} and the last equality follows from Dedekind's modular law as $\psi(L)\le \psi(\mathrm{St}_G(1))$.
\end{proof}

\begin{proof}[Proof of \cref{Theorem: algorithm Hausdorff dimension}]
First note that $K$ may be assumed to be normal by replacing~$K$ with $\mathrm{Core}_G(K)$ as $K$ is of finite index; see \cite{JorgeJoneOihana} for a short proof. We claim that $K\mathrm{St}_G(\Omega(G:K))=K\mathrm{St}_G(\Omega(G:K)+1)$. Thus $K\mathrm{St}_G(n)=K\mathrm{St}_G(n+1)$ for all $n\ge \Omega(G:K)$ by \cref{proposition: if trivial always trivial} which implies that $\overline{K}\ge \mathrm{St}_{\overline{G}}(\Omega(G:K))$ as $\overline{G}$ is profinite with respect to the congruence topology. Therefore $\overline{G}$ is regular branch over its $\Omega(G:K)$th level stabilizer so $s_n(G)=0$ for all $n> \Omega(G:K)$ and the result on the Hausdorff dimension follows from \cref{theorem: sn formula}. Let us prove our claim. Since $\overline{K}$ is open, it contains a level-stabilizer. Hence, by \cref{proposition: if trivial always trivial}, we may assume that there exists $k\ge 1$ maximal such that $K\mathrm{St}_G(n-1)>K\mathrm{St}_G(n)$ for all $n\le k$. We just need to prove that $k\le \Omega(G:K)$. Since
$$G> K\mathrm{St}_G(1)> \dotsb> K\mathrm{St}_G(k)\ge K$$
the orders of the consecutive quotients of the above filtration are non-trivial divisors of the index $|G:K\mathrm{St}_G(k)|$. Then $k$ is at most the number of prime factors with repetition of $|G:K\mathrm{St}_G(k)|$ and we conclude that $k \le \Omega(G:K\mathrm{St}_G(k))\le \Omega(G:K)$, as $|G:K\mathrm{St}_G(k)|$ divides $|G:K|$.

To conclude, we just need to compute $\log|G_n|$ for $n\ge 1$. Since $G$ is regular branch over its $\Omega(G:K)$th stabilizer, the same computations as in the proof of \cite[Proposition 2.7]{BartholdiHausdorff} yield precisely the values of $\alpha$ and $\beta$ in the statement.
\end{proof}

\section{An application: GGS-groups acting on the 4-adic tree}
\label{section: an example}

The \textit{Grigorchuk-Gupta-Sidki groups}, \textit{GGS-groups} for short, acting on the 4-adic tree $T_4$ are the groups generated by $a$ and $b$, where $a$ is the rooted automorphism $\psi(a)=(1,1,1,1)(1\,2\,3\,4)$ and $b$ is the directed automorphism $\psi(b)=(a^{e_1},a^{e_2},a^{e_3},b)$, for some $\mathbf{e}:=(e_1,e_2,e_3)\in (\mathbb{Z}/4\mathbb{Z})^3$. The vector $\mathbf{e}$ is called the \textit{defining vector} of the group. The vector \textbf{e} is said to be \textit{invertible-symmetric}, \textit{IS} for short, whenever $e_i$ is a unit modulo 4 if and only if $e_{4-i}$ is a unit modulo~4 for every $1\le i\le 3$; see \cite{GGSElena}. All these GGS-groups are subgroups of $W_4$, where $W_4:=W_{\langle (1\, 2\, 3\, 4)\rangle}\le \mathrm{Aut}~T_4$. 

\subsection{Hausdorff dimension}
We apply the algorithm in \cref{Theorem: algorithm Hausdorff dimension} to the closure of all of the GGS-groups acting on the $4$-adic tree with a non-IS defining vector.

\begin{theorem}
\label{Theorem: GGS Hausdorff example}
The closure of the GGS-group $G\le \mathrm{Aut}~T_4$ given by a non-IS defining vector $\mathbf{e}=(e_1,e_2,e_3)$ satisfies the following:
\begin{align*}
    \mathrm{hdim}_{W_4}(\overline{G})=\begin{cases}
3/4,&\quad \mathrm{if~} e_1+e_2+e_3\equiv 1\mathrm{~mod~} 2\\
9/16,&\quad \mathrm{otherwise}.
\end{cases}
\end{align*}
What is more, the logarithmic orders of the congruence quotients are
\begin{align*}
    \log|G_n|=\begin{cases}
4^{n-1}+1,&\quad \mathrm{if~} e_1+e_2+e_3\equiv 1\mathrm{~mod~} 2\\
3\cdot 4^{n-2}+1,&\quad \mathrm{otherwise}
\end{cases}
\end{align*}
for $n\ge 1$ and $n\ge 2$ respectively in the formula above.
\end{theorem}
\begin{proof}
    Note that by \cite[Lemma 2.4]{GGSElena} we may assume that either the first or the second component of the defining vector is equal to 1, as for any $f\in \mathrm{Aut}~T_4$ we have that $\log|G^f_n|=\log|G_n|$ for all $n\ge 1$ and a permutation of the defining vector does not affect the statement of \cref{Theorem: GGS Hausdorff example}. Since we are dealing with non-IS defining vectors it cannot be that only the second component is invertible modulo~4, thus we may always assume that the first component is 1. Note that $G$ is regular branch over $G'$ by \cite[Theorem 2.7]{GGSElena} and we have $|G:G'|=4^2$ by \cite[Theorem 2.1]{GGSElena}. Therefore $\Omega(4^2)=4$, so $s_n(G)=0$ for all $n\ge 5$ by \cref{Theorem: algorithm Hausdorff dimension}. Hence we only need to compute $\log|G_n|$ for $n=1,2,3,4,5$ and apply the formulae in \cref{Theorem: algorithm Hausdorff dimension}. This is done in GAP for all the 8 different non-IS defining vectors whose first entry equals to 1; see \cref{table: GAP data} below. 
\end{proof}

\subsection{Discarding branch structures}

The proof of \cref{Theorem: algorithm Hausdorff dimension} also suggests a method to discard possible branch structures in a self-similar group acting faithfully on the $m$-adic tree. If we consider a branching subgroup $K$ of finite index in $G$ its closure is open in $\overline{G}$ so $\overline{K}\ge \mathrm{St}_{\overline{G}}(k)$ for some $k\ge 1$. Thus we have $s_n(G)=0$ for all $n>k$. If we have $s_n(G)\ne 0$ for some $n>k$ then we can conclude $G$ is not regular branch over $K$.

As an example let us consider the second Grigorchuk group $G$, i.e. the GGS-group $G\le \mathrm{Aut}~T_4$ given by the defining vector $\mathbf{e}=(1,0,1)$. It is known that $G$ is regular branch over $\gamma_3(G)$ \cite[Proof of Lemma 2.1]{Pervova}. We discard a finer branch structure over $G'$.

\begin{proposition}
\label{proposition: Second Grigorchuk example}
The second Grigorchuk group $G$  does not branch over~$G'$.
\end{proposition}
\begin{proof}
    It was proved by Pervova in \cite[Lemma~3.2]{Pervova} that $G'\ge \mathrm{St}_G(2)$. Therefore, were it to branch over $G'$, we would have $s_n(G)=0$ for all $n\ge 3$ by the discussion above. However, a computation in GAP shows that
\begin{align*}
    \log|G_1|=1,~\log|G_2|=3,~\log|G_3|=8.5 \text{ and }\log|G_4|=30
\end{align*}
yielding $s_1(G)=2, s_2(G)=2.5$ and $s_3(G)=0.5\ne 0$ , which results in a contradiction.
\end{proof}
 
However as $\gamma_3(G)\ge \mathrm{St}_G(3)$ by \cite[Lemma 3.2]{Pervova} and $G$ branches over $\gamma_3(G)$, the above computations are enough to compute the Hausdorff dimension of the closure of the second Grigorchuk group and we obtain
$$\mathrm{hdim}_{W_4}(\overline{G})=1-S_G(1/4)=1-\frac{2}{4}-\frac{2.5}{4^2}-\frac{0.5}{4^3}=\frac{43}{128},$$
which agrees with the result of Noce and Thillaisundaram who first computed the Hausdorff dimension of the closure of the second Grigorchuk group \cite[Theorem~A(ii)]{MarialauraAnitha}.

\begin{table}[H]
\begin{center}
\begin{tabularx}{\textwidth}[t]{XXXX}
\arrayrulecolor{teal}\hline
\rowcolor{lightBlue} \textcolor{teal}{Defining vector} & \textcolor{teal}{$(\log|G_n|)_{n=1}^5$} & \textcolor{teal}{$(s_n(G))_{n=1}^4$} & \textcolor{teal}{$\mathrm{hdim}_{W_4}(\overline{G})$} \\
\hline

$(1,0,0)$ & $(1,5,17,65,257)$ & $(0,4,0,0)$ & $3/4$\\
$(1,0,2)$ & $(1,5,17,65,257)$ & $(0,4,0,0)$ & 3/4\\
$(1,2,0)$ & $(1,5,17,65,257)$ & $(0,4,0,0)$ & 3/4\\
$(1,2,2)$ & $(1,5,17,65,257)$ & $(0,4,0,0)$ & 3/4\\

\arrayrulecolor{mygray}\hline

$(1,1,0)$ & $(1,4,13,49,193)$ & $(1,3,0,0)$ & 9/16\\
$(1,1,2)$ & $(1,4,13,49,193)$ & $(1,3,0,0)$ & 9/16\\
$(1,-1,0)$ & $(1,4,13,49,193)$ & $(1,3,0,0)$ & 9/16\\
$(1,-1,2)$ & $(1,4,13,49,193)$ & $(1,3,0,0)$ & 9/16\\

\hline
\end{tabularx}
\end{center}
\caption{GAP data for the GGS-groups with a non-IS defining vector acting on the 4-adic tree.}
\label{table: GAP data}
\end{table}



\bibliographystyle{unsrt}

\begin{thebibliography}{1}

\bibitem{Abercrombie}
A.\,G. Abercrombie, Subgroups and subrings of profinite rings, \textit{Math. Proc. Camb. Phil. Soc.}, \textbf{116 (2)} (1994), 209--222.


\bibitem{BarneaShalev}
Y. Barnea and A. Shalev, Hausdorff dimension, pro-$p$ groups, and Kac-Moody algebras, \textit{Trans. Amer. Math. Soc.}, \textbf{349} (1997), 5073--5091.

\bibitem{BartholdiHausdorff}
L. Bartholdi, Branch rings, thinned rings, tree enveloping rings, \textit{Israel J. Math.} \textbf{158} (2006), 93--139.

\bibitem{Handbook}
L. Bartholdi, R.\,I. Grigorchuk and Z. \v{S}uni\'{c}, Branch groups, in: \textit{Handbook of algebra}, \textbf{3}, North-Holland, Amsterdam, 2003.


\bibitem{MarialauraBartholdi}
L. Bartholdi and M. Noce, Tree languages and branched groups, \textit{Math. Z.}, \textbf{303} (2023), 96.


\bibitem{GGSElena}
E. Di Domenico, G.\,A. Fernández-Alcober and N. Gavioli, GGS-groups over primary trees: branch structures, \textit{Monatsh. Math.}, \textbf{200} (2022), 781--797.


\bibitem{Falconer}
K. Falconer, \textit{Fractal Geometry: mathematical foundations and applications}, John Wiley \& Sons, New York, 1990.  

\bibitem{JorgeCyclicity}
J. Fariña-Asategui. Cyclicity, hypercyclicity and randomness in self-similar groups, \textit{Monatsh. Math.}, to appear.

\bibitem{JorgeSpectra}
J. Fariña-Asategui, Restricted Hausdorff spectra of $q$-adic automorphisms, arXiv preprint: 2308.16508.

\bibitem{JorgeJoneOihana}
J. Fariña-Asategui, O. Garaialde Ocaña and J. Uria-Albizuri, Hausdorff spectra in branch groups, in preparation.

\bibitem{SukranJone}
Ş. Gül and J. Uria-Albizuri, Grigorchuk–Gupta–Sidki groups as a source for Beauville surfaces, \textit{Groups Geom. Dyn.}, \textbf{14 (2)} (2020), 689--704.

\bibitem{Crypto}
D. Kahrobaei, R. Flores, M. Noce, M.\,E. Habeeb, and C. Battarbee, \textit{Applications of group theory in cryptography: post-quantum group-based cryptography}, volume 278 of \textit{Mathematical Surveys and Monographs}, American Mathematical Society, 2024.

\bibitem{SelfSimilar}
V. Nekrashevych, \textit{Self-similar groups}, volume 117 of \textit{Mathematical Surveys and Monographs}, American Mathematical Society, Providence, RI, 2005.

\bibitem{MarialauraAnitha}
M. Noce and A. Thillaisundaram, Hausdorff dimension of the second Grigorchuk group, \textit{Internat. J. Algebra Comput.}, \textbf{31} (2021), 1037--1047.

\bibitem{PenlandSunicSubshifts}
A. Penland and Z. \v{S}uni\'{c}, A language hierarchy and Kitchens-type theorem for self-similar groups, \textit{J. Algebra}, \textbf{537} (2019), 173--196.


\bibitem{Pervova}
E.\,L. Pervova, Profinite topologies in just infinite branch groups, preprint 2002-154 of the Max Planck Institute for Mathematics, Bonn, Germany.

\bibitem{SunikHausdorff}
Z. \v{S}uni\'{c}, Hausdorff dimension in a family of self-similar groups, \textit{Geom. Dedicata}, \textbf{124} (2007), 213--236.






\end{thebibliography}

\end{document}